\documentclass[reqno, a4paper]{amsart}
\usepackage{amsmath, amssymb, amsthm, enumitem}
\usepackage{mathtools}
\usepackage{placeins}
\usepackage{subcaption,float}

\pdfoutput=1
\usepackage{color}
\def\comment#1{}

\usepackage{hyperref}

\usepackage{txfonts}

\usepackage{tikz}
\usetikzlibrary{external}
\usetikzlibrary{backgrounds,patterns,decorations.pathreplacing,decorations.pathmorphing}
\usepackage{caption}
\usepackage{accents}
\usepackage{mathrsfs}

\newtheorem{theorem}{Theorem}

\newtheorem{lemma}[theorem]{Lemma}

{Important Convention}

\theoremstyle{remark}

 \renewcommand{\phi}{\varphi}

\newcommand{\E}{\mathbb{E}}

\newcommand{\Q}{\mathbb{Q}}
\newcommand{\R}{\mathbb{R}}

\newcommand{\bes}{\begin{subequations}}
\newcommand{\ees}{\end{subequations}}
\newcommand{\eea}{\end{eqnarray}}

\newcommand{\NN}{{\mathbb N}}

\newcommand{\EE}{{\mathbb E}}

\newcommand{\WW}{{\mathbb W}}

\usepackage{bbm}

\renewcommand{\epsilon}{\varepsilon}

\newcommand{\fourIdx}[5]{%
\setbox1=\hbox{\ensuremath{^{#1}}}%
 \setbox2=\hbox{\ensuremath{_{#2}}}%
 \setbox5=\hbox{\ensuremath{#5}}%
 \hspace{\ifnum\wd1>\wd2\wd1\else\wd2\fi}%
 \ensuremath{\copy5^{\hspace{-\wd1}\hspace{-\wd5}#1\hspace{\wd5}#3}%
 _{\hspace{-\wd2}\hspace{-\wd5}#2\hspace{\wd5}#4}%
 }}

\numberwithin{equation}{section}
\numberwithin{theorem}{section}

\newcommand{\mo}{\bar{M}}
\newcommand{\vo}{\bar{v}}
\newcommand{\so}{\bar{\sigma}}

\usepackage{subcaption}
\usepackage{graphicx}

\renewcommand{\mathrm}{}

\newcommand{\mylabel}[2]{#2\def\@currentlabel{#2}\label{#1}}

\begin{document}

\title[The most exciting game]{The most exciting game}

\author{Julio Backhoff-Veraguas and Mathias Beiglb\"ock}\thanks{We thank David Aldous for pointing us to the problem posed in  \cite{Al22a, Al22b} and suggesting the connection to specific relative entropy. We also acknowledge support by the Austrian Science Fund FWF through projects P36835, Y0782, and P35197 }

\begin{abstract} Motivated by a problem posed by Aldous \cite{Al22a, Al22b} our goal is to find the maximal-entropy win-martingale: 

In a sports game between two teams, the chance the home team wins is initially $x_0 \in (0,1)$
 and finally 
0 or 1. As an idealization we take a continuous time interval $[0,1]$ and consider the process $M=(M_t)_{t\in [0,1]}$
 giving the probability at time 
$t$ that the home team wins.  This is a martingale which we idealize further to have continuous paths. 
We consider the problem to find the most random martingale $M$ of this type, where `most random' is interpreted as a maximal entropy criterion.  We observe that this max-entropy win-martingale $M$ also minimizes specific relative entropy with respect to Brownian motion in the sense of Gantert  \cite{Ga91} and use this to prove that $M$ is characterized by the stochastic differential equation
$$ dM_t = \frac{\sin (\pi M_t )} {\pi\sqrt {1-t}}\, dB_t.$$
To derive the form of the optimizer we use a scaling argument together with a new first order condition for martingale optimal transport which may be of interest in its own right. 
\end{abstract}

\maketitle

\section{Introduction}

\subsection{Main result}
We  write  $\mathcal M^c_{x_0}$ for the set of laws of continuous martingales  with time-index set $[0,1]$ which  have  absolutely continuous quadratic variation and start in $x_0$. The subset $\mathcal M^c_{x_0, \text{win}}$ of  \emph{win-martingales}   consist of those martingales which  terminate in either $0$ or $1$. Win-martingales appear naturally as models for prediction markets (cf.\ \cite{Al13}). 

In our main result we characterize the win-martingale which is closest to Brownian motion in that it minimizes the \emph{specific relative entropy $h$}     (in the sense of Gantert \cite{Ga91}) w.r.t.\ Wiener measure $\mathbb{W}^{x_0}$ started at $x_0$. 
\begin{theorem}
Let $x_0\in (0,1)$. Then the minimization problem
\begin{align}\inf\{h(\mathbb{Q}| \mathbb{W}^{x_0}): \Q\in \mathcal M^c_{x_0, \text{win}} \} \label{eq:max_ent_prob}
\end{align}
has a unique solution which is given through the SDE
\begin{align}\label{eq:AldousMart} dM_t = \frac{\sin (\pi M_t )} {\pi\sqrt {1-t}}\, dB_t, \quad M_0=x_0.\end{align}
\end{theorem}
Gantert \cite{Ga91} defines specific relative entropy $h$ as the scaled limit of relative entropies associated to appropriate time-discretizations of $\Q$ and $\mathbb{W}^{x_0}$. For sufficiently regular $\Q$, it is known (cf.\ \cite{Ga91, BaUn22}) that  
\begin{align}\label{eq:specific_entropy} \textstyle h(\mathbb{Q}| \mathbb{W}^{x_0})= \frac{1}{2}\E_{\Q}\left[ \int_0^1\{\Sigma_t -\log(\Sigma_t) -1\}\,  dt \right ],\end{align}
where $\Sigma_t$ stands for the density of the quadratic variation for the canonical process under $\Q$  at time $t$ (see \cite{Ga91, Fo22a}). 
For us it is important  to directly use the quantity in \eqref{eq:specific_entropy}, we discuss this in Section 3 below.

Aldous \cite{Al22a, Al22b} considers the problem of determining the win martingale which maximizes entropy. To make this precise  Aldous considers classical Shannon entropy and win martingales in a discretized setting. He provides heuristics for the existence of a scaling limit and  suggests a PDE that should be satisfied by this limit. In this paper we observe that (discrete time)  martingales which  \emph{maximize} Shannon entropy \emph{minimize} relative entropy w.r.t.\ discretized Wiener measure as used by Gantert \cite{Ga91}.  
This observation  indicates that \eqref{eq:max_ent_prob} is an appropriate way to define the max-entropy win-martingale directly in continuous time. Indeed, we   verify that the martingale $M$ specified in  \eqref{eq:AldousMart} solves the PDE suggested by Aldous and we refer to $M$ as \emph{Aldous martingale}. 


%

\begin{figure}
\centering
\begin{minipage}{.45\textwidth}
\centering 
\includegraphics[width=.7\linewidth]{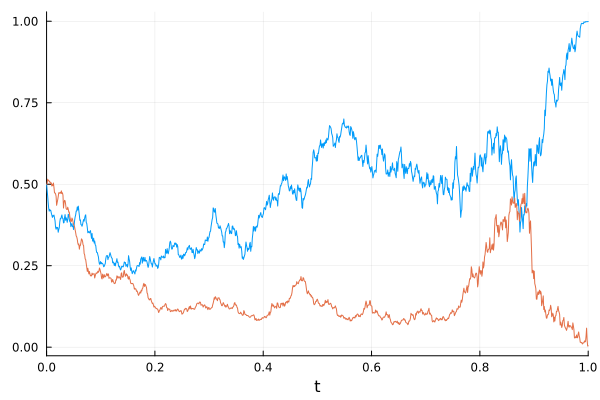}
\caption{Two typical paths of the Aldous martingale for a fair game ($M_0=0.5$).}
  \label{fig:Aldous}
\end{minipage}%
\begin{minipage}{.45\textwidth}
\centering
\includegraphics[width=.7\linewidth]{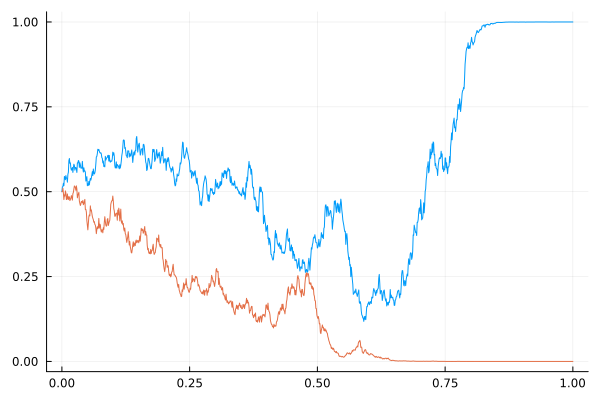}
\caption{Two typical paths of the Bass martingale $X_t:=\Phi(B_t/\sqrt{1-t})$ from $0.5$ to $Bernoulli(0.5)$.}
\label{fig:Bass}
\end{minipage}
\end{figure}

\begin{figure}
\centering
\begin{minipage}{.45\textwidth}
\centering
\includegraphics[width=.7\linewidth]{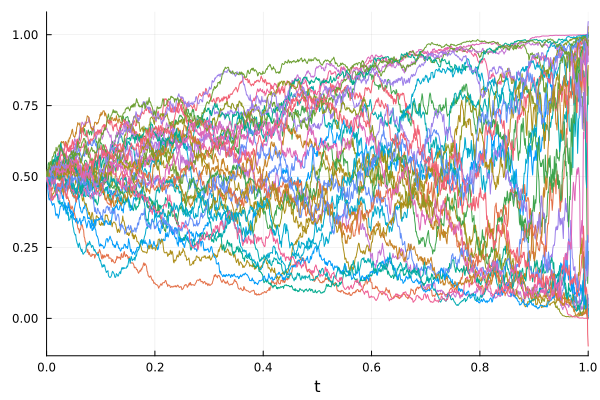}
\caption{Many simulated paths of the Aldous martingale for a fair game.}
  \label{fig:Aldous2}
\end{minipage}%
\begin{minipage}{.45\textwidth}
\centering
\includegraphics[width=.7\linewidth]{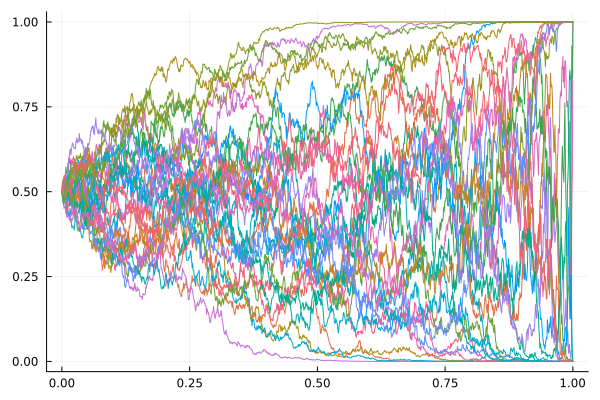}
\caption{Many simulated paths of the Bass martingale.}
\label{fig:Bass2}
\end{minipage}
\end{figure}
In Figures \ref{fig:Aldous} to \ref{fig:Bass2} we compare simulations of the Aldous and the Bass martingale. The Bass martingale was introduced in \cite{Ba83,BaBeHuKa20, BaBeScTs23} an is the martingale with specified initial and terminal distribution which minimizes the \emph{adapted Wasserstein distance} to Brownian motion. 

Note added in revision: briefly after making our note available to the arxiv, and independently from us, Guo, Possama\"i and Reisinger presented a different  solution to Aldous' problem based on PDE-techniques and unrelated to the concept of specific entropy. Importantly, while the article of Guo, Possama\"i and Reisinger \cite{GuPoRe23} as well as the present note yield solutions of the same PDE, the respective solutions satisfy different boundary conditions.

\subsection{Outline of the paper}

In Section 2 we will discuss martingale transport problems in discrete time which optimize an entropy criterion. In particular we detail that maximization of Shannon entropy corresponds precisely to minimization of relative entropy with respect to discretized Wiener measure.

In Section 3 we discuss specific relative entropy in the sense of Gantert \cite{Ga91}. We also observe that the Aldous martingale provides a counterexample to a conjecture posed in \cite{Ga91} about the representation of specific relative entropy.

In Section 4 we derive a new first order condition for martingale transport problems in continuous time. Together with a scaling argument for optimizers of \eqref{eq:max_ent_prob} this allows to identify  the Aldous martingale \eqref{eq:AldousMart} as a candidate optimizer. 

Finally, in Section 5  we use verification arguments to obtain that \eqref{eq:AldousMart}  indeed solves \eqref{eq:max_ent_prob}. In particular, the optimal value of \eqref{eq:max_ent_prob} is equal to
$$\frac{x_0(1-x_0)-1}{2}-\log\left(\frac{\sin(\pi x_0)}{\pi} \right ).$$


\section{Maximization of  entropy and minimization of relative entropy}

As noted above, 
Aldous  \cite{Al22a, Al22b}   poses the problem to determine the \emph{maximal-entropy win-martingale}. To assign a rigorous meaning to  the maximal entropy (in the sense of Shannon ) criterion, his starting point is a formulation in a discrete time, discrete space setting. It is then argued, that a natural scaling limit should exist. 
The goal of this section is to explain the connection between maximization of entropy and minimization of relative entropy  w.r.t.\  the Gaussian random walk / discritized Wiener measure as used in Gantert's definition of specific relative entropy. We emphasize that this section is technically  not required for the  results in this article. Rather the point we want to make is that the notion of  specific relative entropy allows to pose the problem directly in the continuous time setup. 
 
We  deviate from the  setting of \cite{Al22a, Al22b} in that we  consider discrete time (for now)  martingales which are allowed to  take values on the real line. We also allow for general  prescribed initial and terminal distributions $\mu, \nu$  which are in convex order.
For $T\in \NN$ we  write $(X_t)_{t=0}^T$ for the canonical process on the path space $\R^{T+1}$. We denote by  $M_T(\mu, \nu) $  the set of all martingale transport plans, that is, martingale measures $\Q$ on $\R^{T+1}$ such that $X_0(\Q)= \mu, X_T(\Q)=\nu$. We then consider the maximal entropy transport problem 
\begin{align}\label{eq:DiscreteMaxEntropy}
\max_{\Q\in M_T(\mu, \nu)} H(\Q) \quad
\text{where } H(\Q) = \begin{cases} - \int_{\R^{T+1}} q(x) \log q(x)\, dx   & \Q= q \cdot \lambda_{T+1} \\
-\infty & \text{else} \end{cases},
\end{align} 
where $\lambda_d$ denotes Lebesgue measure on $\R^d$.
Note that \eqref{eq:DiscreteMaxEntropy} can be finite only if $\mu, \nu$ are absolutely continuous, which we will  tacitly assume. 
We write $$\gamma_{T}(x_0, \ldots, x_T) \propto  f_0 (x_0)  \exp{\left(\frac1{2\sigma^2}\sum_{i=1}^T (x_i-x_{i-1})^2 \right)} 
$$
for the density of the Gaussian random walk with start in $f_0 \cdot \lambda_1$. 
For $\Q= q \cdot \lambda_{T+1}\in M_T(\mu, \nu)$ the relative entropy $H(\Q|\gamma_T) $ satisfies
\begin{align*}
H(\Q|\gamma_T) & := \int \log \frac q{\gamma_T} \,  d \Q
= \int q(x) \log q(x)\, dx  -  \int \log \gamma_T  d\Q=  
\\
& =  - H(\Q)  +\frac  T 2 \log (2\pi \sigma^2 ) - \int \log f_0 \, d\mu + \int \sum_{t=1}^T\frac{(x_i-x_{i-1})^2} {2 \sigma^2} \, d \Q(x_0, \ldots, x_T) \\ 
&=      - H(\Q)  + \frac  T 2 \log (2\pi \sigma^2 )  - \int \log f_0 \, d\mu + \frac1{2\sigma^2}  \left( \int x_T^2 \, d\nu(x_T) - \int x_0^2 \,  d\mu(x_0)  \right) .
\end{align*}
We thus find that, up to additive constants, maximization of Shannon entropy in \eqref{eq:DiscreteMaxEntropy} corresponds precisely to minimization of relative entropy w.r.t.\ the Gaussian random walk. Note that the choice of $f_0$ is, up to integrability issues, not important for this argument.


\section{On specific relative entropy}

We write $X=(X_t)_{ t \in [0,1]} $ for the canonical process on Wiener space $C=C([0,1])$, and also consider the discretized process
$$X^n:=(X_{k/n})_{k=0, \ldots, n}.$$ We further consider  a probability $\Q$ and denote the Wiener measure with start in $x_0$ by $\WW^{x_0}$. Gantert \cite{Ga91} defines the specific relative entropy of $\Q$ w.r.t.\ 
$\WW^{x_0}$ as the limit of scaled relative entropies of the discretization of $\Q$  w.r.t.\ the Gaussian random walk, i.e.\
\begin{align}\label{eq:GantertDef}
h(\Q|\WW^{x_0})=\lim_{n\to \infty} \frac{1}n H(X^n(\Q)|X^n(\WW^{x_0})) 
\end{align}
whenever this limit exists. 

The specific relative entropy is meaningful even in situations where measures singular to each other are being compared. This is the case of continuous martingale laws, which typically have infinite relative entropy but may still have a finite specific relative entropy. Gantert \cite[Kapitel II.4]{Ga91} shows that $h$ is the rate function in a large deviations principle associated to a randomized Donsker-type approximation of Brownian motion. The specific relative entropy is also studied by F\"ollmer \cite{Fo22a, Fo22b} who uses it to establish Talagrand-type inequalities on the Wiener space beyond the absolutely continuous case. In particular he proves that the squared adapted Wasserstein distance between a continuous martingale and Brownian motion is bounded from above by twice the specific relative entropy. {The quantity \eqref{eq:GantertDef} is also considered independently by Cohen and Dolinsky \cite{CoDo22}, where it plays a role in the derivation of the  scaling limit of utility indifference prices. On the other hand, formula \eqref{eq:specific_entropy}, as well as similar expressions, appeared in the work of Avellaneda, Friedman, Holmes and Samperi \cite{Av01} concerning model calibration in finance.}

According to \cite{Ga91} (in particular situations) and to \cite{BaUn22} (for all sufficiently regular martingale diffusions), we have the alternative expression
\begin{align}\label{eq:IntDef}
 h(\mathbb{Q}| \mathbb{W}^{x_0})= \frac{1}{2}\E_{\Q}\left[ \int_0^1\{\Sigma_t -\log(\Sigma_t) -1\}\,  dt \right ],
\end{align}
where $\Sigma_t$ stands for the density of the quadratic variation at time $t$ for the canonical process under $\Q$. More generally \cite{Ga91, Fo22a} 
show that  
\begin{align} \label{eq:alt_defi}
\liminf_{n\to \infty} \frac{1}n H(X^n(\Q)|X^n(\WW^{x_0})) \geq \frac{1}{2}\E_{\Q}\left[ \int_0^1\{\Sigma_t -\log(\Sigma_t) -1\}\,  dt \right ],
\end{align}
for all martingales, and conjecture that equality should hold for all martingales (with absolutely continuous quadratic variation). We will see in Section 5 that the Aldous martingale has finite specific entropy when using the representation in the r.h.s.\ of \eqref{eq:IntDef}. However $  H(X^n(\Q)|X^n(\WW^{x_0})) =\infty $ for all $n\in \NN$  for \emph{any} win martingale, since $X_1(\Q)\in \{0,1\}$ $\Q$-a.s. For this reason \eqref{eq:IntDef} fails even in the class of time-inhomogeneous diffusions, and it is crucial that we use the representation in \eqref
{eq:specific_entropy} / \eqref{eq:IntDef}  as our definition of specific relative entropy.

\section{First order conditions for martingale optimal transport}

Following \cite{HoNe12, BeHePe12, BoNu13,GaHeTo13}
martingale versions of the classical transport problem (see e.g.\ \cite{Vi03, Vi09, Sa15,FiGl21} for recent monographs) are often considered due to applications in  mathematical finance but admit further applications, e.g.\ to the Skorokhod problem \cite{BeCoHu14, BeNuSt19}. In analogy to classical optimal transport, necessary and sufficient conditions for optimality have been established for martingale transport (MOT) problems in discrete time (\cite{BeJu21, BeNuTo16}) but not so much is known for the continuous time problem. 

We derive here a new first order condition for martingale transport in continuous time which we then use to determine the structure of a candidate optimizer for \eqref{eq:max_ent_prob}. Specifically,  we consider the following MOT problem
%
\begin{align}\label{def:MOT}
\inf\left\{\E_{\Q}\left[ \int_0^1 c\left(t,X_t, \Sigma_t \right) dt \right ]: \Q\in \mathcal M^c([0,1]),\,X_0(\Q)=\mu, X_1(\Q)=\nu \right\},
\end{align}
where $\mathcal M^c([0,1])$ denotes the set of martingale measures on $C$ with an absolutely continuous quadratic variation,  we  denote by $X$ the canonical process and by $\Sigma$ the density of its quadratic variation. As above, we assumes $\mu\leq\nu$ in convex order, so that  $\mathcal M^c([0,1])$ is non empty.


\begin{lemma}\label{eq:FirstOrderMOT}[First order condition for MOT] 
Consider the MOT problem \eqref{def:MOT}, and suppose that $c$ is differentiable in its last variable, that $\Q$ is an optimizer, and that 
$$t\mapsto L_t:= \Sigma_t\partial_\Sigma c(t,X_t,\Sigma_t)-c(t,X_t,\Sigma_t),$$ is a continuous $\Q$-semimartingale. Then $(L_t)_{t\in[0,1)}$ is a  martingale under $\Q$.
\end{lemma}

Of course if $c=c(\Sigma)$ is either convex or concave one can write $xc'(x)-c(x)=c^*\circ c'(x)$ {by the definition of the convex/concave conjugate $c^*$}.

\begin{proof} Consider now $H$ an absolutely continuous adapted process satisfying
$$H_0=0=H_1\,\,\,\text{ and }\,\,\, h_t:=H_t' \in (-1,1).$$
It follows that for all $0\leq \epsilon\leq 1$ the function
$$\tau^\epsilon_t= t+\epsilon H_t,$$
defines a time change. Namely, this is a continuous  increasing adapted process  starting at $0$ and ending at $1$. The time-changed martingale $\omega^\epsilon_t:=X_{\tau^\epsilon_t}$ has the same starting and final marginals, and so it must be suboptimal for the MOT problem. Since the density of the quadratic variation of $\omega^\epsilon_t$ is $$\Sigma_{\tau^\epsilon_t}\cdot(\tau^\epsilon_t)',$$ we compute that the cost associated to $\omega^\epsilon_t$ is
$$\EE_\Q\int_0^1 c\left (\tau^\epsilon_t,\omega^\epsilon_t,\Sigma_{\tau^\epsilon_t}\cdot(\tau^\epsilon_t)'\right )dt = \EE_\Q\int_0^1 c\left (s,X_s,\Sigma_{s}\cdot\{(\tau^\epsilon_\cdot)'\circ (\tau^\epsilon_s)^{-1}\right \})\frac{ds}{(\tau^\epsilon_\cdot)'\circ (\tau^\epsilon_s)^{-1}}.$$
Since $(\tau^\epsilon_\cdot)'\circ (\tau^\epsilon_s)^{-1}=1+h_{(\tau^\epsilon_s)^{-1}}$, we have
$$\EE_\Q\int_0^1 c\left (\tau^\epsilon_t,\omega^\epsilon_t,\Sigma_{\tau^\epsilon_t}\cdot(\tau^\epsilon_t)'\right )dt =   \EE_\Q\int_0^1 c\left (s,X_s,\Sigma_{s}\{1+h_{(\tau^\epsilon_s)^{-1}}\}\right )\frac{ds}{1+h_{(\tau^\epsilon_s)^{-1}}}.$$
Then a few computations reveal that
$$\frac{d}{d\epsilon}\Big\vert_{\epsilon=0}\,\EE_\Q\int_0^1 c\left (\tau^\epsilon_t,\omega^\epsilon_t,\Sigma_{\tau^\epsilon_t}\cdot(\tau^\epsilon_t)'\right )dt=\EE_\Q\int_0^1\{-c(s,X_s,\Sigma_s)+ \partial_\Sigma c(s,X_s,\Sigma_s)\Sigma_s\}h_sds =0,$$
for all $h=H'$ as above. From this, and integration by parts, we have that  $\EE_\Q\int_0^1 H_t dL_t$ is equal to zero, for all $H$ as above and with $L$ as in the statement of this lemma. We conclude that $L$ must be a martingale for $t\in[0,1)$.
\end{proof}

Going back to \eqref{eq:max_ent_prob}, that is,  $c(\Sigma)=\frac{1}{2}[\Sigma-\log(\Sigma)-1]$, $\mu=\delta_{x_0}$, $\nu=x_0\delta_1+(1-x_0)\delta_0$, the previous lemma shows that if $\Q$ is optimal then \[t\mapsto \log(\Sigma_t),\]
 is a martingale under $\Q$. Passing to the notation where $\Sigma=\sigma^2$ denotes the density of the quadratic variation, we also conclude that 
\[t\mapsto \log(\sigma_t),\]
is a martingale (we convene that $\sigma$ is the positive square root of $\sigma^2$). If the coefficient $\sigma$ is Markovian, then the process $\log(\sigma_t)$ can only be a martingale if
\begin{align}
\left(\partial_t+\frac{1}{2}\sigma^2\partial^2_{xx}\right)\log\circ\sigma =0, \label{eq:foc_FP1}
\end{align}
We emphasize that this is exactly the PDE suggest by Aldous \cite{Al22a} for the scaling limit of maximum entropy win martingales.

We will use this piece of information to obtain a candidate optimizer for the problem at hand.
To obtain further intuition on the structure of $\sigma$, note the goal is to find the win martingale  as close as possible to Brownian motion, formalized  by an entropy criterion. By the chain rule of entropy, this suggests, independently of the behavior of the optimizer $M$ before it reaches a level $m$ at time $t$, that the behavior from that time point on should be optimal again. First of all, this indicates that the optimal $\sigma$ is Markovian/feedback, so 
$$dM_t=\sigma(t,M_t)dBt.$$ 
In fact, due to special nature of the constraint to be a win martingale, we can make even stronger guesses concerning the structure of $\sigma$: If we start the optimization problem at time $t$ at level $M_t=x$ we face again the problem to be as close as possible to Brownian motion subject to being a martingale which terminates at time $1$ in a Bernoulli distribution. That is, we face the very same optimization problem as at time $0$ apart from the fact that the time horizon is now $1-t$ rather than $1$ in the original problem. We thus expect the optimizer to be the same as the original optimizer, scaled by the factor $1-t$, that is, run at a speed which is higher by a factor of $\frac 1{1-t}.$ As volatility is the square root of quadratic variation this amounts to  
\begin{align}\label{eq:VeryEducatedGuess}
\sigma(t, x) = \frac1 {\sqrt{1-t}} \sigma(1,x).
\end{align}
Writing $\sigma(x):= \sigma(1,x)$ and plugging \eqref{eq:VeryEducatedGuess}  into \eqref{eq:foc_FP1} yields
\begin{align*}
 & \left(\partial_t+\frac{\sigma^2(x)}{2 (1-t)}\partial^2_{xx}\right)\left(  -\frac 12 \log(1-t)  + \log(\sigma(x)) \right) =0,\\
\Leftrightarrow \quad  & -\frac 12 \frac1{1-t} = \frac 12 \frac1{1-t} \sigma^2(x) \left(\frac{ \sigma''(x) \sigma(x) - (\sigma'(x))^2}{\sigma(x)^2} \right)\\
\Leftrightarrow \quad &\hspace{8mm} -1 = \sigma''(x) \sigma(x) - (\sigma'(x))^2.
\end{align*}
This ODE is solved by $\sigma(x)= \frac 1 \alpha \sin(\alpha x + \beta)$ for real constants $\alpha, \beta$. As the optimizer is supposed to terminate in $0$ or $1$, the only reasonable choice here is $\alpha= \pi, \beta = 0$. 
%
%
%
%
All in all, our Ansatz for the optimal $\sigma$ is
\[\so(t,x):=\frac{\sin(\pi x)}{\pi\sqrt{1-t}}.\]

\section{Optimality of the Ansatz}

We denote by $\mo^{s,x}$ the martingale which starts from time $s\leq 1$ at the position $x\in[0,1]$ and is determined by $\so$. On $[s,1)$ it satisfies the SDE
\[d\mo^{s,x}_t = \frac{\sin(\pi \mo^{s,x}_t)}{\pi\sqrt{1-t}}dB_t\]
We remark that for all $\epsilon$ small enough the coefficient $\so$ is smooth and has bounded derivatives of all orders for $t\in(s,1-\epsilon)$. In particular the above SDE admits a unique strong solution on $[s,1)$. Observe that, for $x\in\{0,1\}$, if $M^{s,x}_\ell = 0$ then also $M^{s,x}_t = 0$ for all $t\in(\ell,1)$. In particular then we have $0\leq\inf_{t\in[s,1)}\mo_t^{s,x}\leq \sup_{t\in[s,1)}\mo_t^{s,x}\leq 1$ a.s. Hence the martingale is bounded in $L^p$ for every $p$ and in particular $\mo^{x,s}_1:= \lim_{t\to 1} \mo^{x,s}_t$ exists a.s.\ and in $L^2$. Thus $\mo^{x,s}_1\in [0,1]$ and  $\mathbb E[\langle \mo^{x,s}\rangle_1]<\infty$, hence also 
$$\mathbb E\left[\int_s^1 \frac{\sin^2(\pi \mo^{s,x}_t)}{1-t}dt \right ]<\infty,$$
and in particular $\int_s^1 \frac{\sin^2(\pi \mo^{s,x}_t)}{1-t}dt <\infty$ a.s. We conclude that the event $\{\mo^{s,x}_1\in (0,1)\}$ is negligible since on this event $\int_s^1 \frac{\sin^2(\pi \mo^{s,x}_t)}{1-t}dt=+\infty$.

We summarise this discussion:

\begin{lemma}
	$\mo^{s,x}$ is well-defined on the whole interval $[s,1]$, it is a continuous martingale bounded in every $L^p$, and it satisfies $\mo^{s,x}_1\in\{0,1\}$ a.s. (implying that $\mo^{s,x}_1\sim Bernoulli(x)$).
\end{lemma}

In fact we have a bit more:

\begin{lemma}\label{lem:FellerTest}
	Let $x\in(0,1)=:I$. If $\tau$ denotes the first time that $\mo^{s,x}$ exits $I$, then $\tau=1$ a.s.
\end{lemma}

\begin{proof}
Wlog $s=0$.	We define, for $t\in\mathbb R_+$, the martingale $M_t :=\mo^{0,x}_{1-\exp(-t)}$.
Hence
\[\langle M\rangle_t = \int_0^{1-\exp(-t)}\frac{\sin^2(\pi\mo^{0,x}_u)}{\pi^2(1-u)}du = \int_0^t\frac{\sin^2(\pi M_s)}{\pi^2}ds,
\]
where we employed the change of variables $s=-\log(1-u)$. Hence $M$ satisfies the SDE on $\mathbb R_+$
\[dM_t = \frac{\sin(\pi M_t)}{\pi}dW_t, \]
with $M_0=x\in I$.
Observe that $\sigma^2(z):=\frac{\sin^2(\pi z)}{\pi^2}>0$ for $z\in I$, and that $1/\sigma^2(z)$ is upper bounded on every compact interval contained in $I$. Hence we can apply Feller's test for explosions (see e.g.\  \cite[Theorem 5.29]{KaSh12}), according to which $M$ does not leave $I$ in finite time (a.s.) if and only if the function
\[V(y):=\int_{1/2}^y\frac{y-z}{\sin^2(\pi z)}dz
\]
satisfies $V(0+)=V(1-)=\infty$.  The latter is clearly fulfilled.
\end{proof}


\begin{lemma}\label{lem_log_is_mart}
	The process $L_t:=\log\frac{\sin(\pi\mo^{s,x}_t)}{\pi\sqrt{1-t}}$ is a martingale on $[0,1)$.
\end{lemma}

\begin{proof}
	Let $f(t,z):= \log\frac{\sin(\pi z)}{\pi\sqrt{1-t}} $ so $L_t=f(t,\mo^{s,x}_t)$. Hence $\partial_z f = \frac{\pi \cos(\pi z)}{\sin(\pi z)} $ and so
	\[\int_s^{s'} \partial_z f(t, \mo^{s,x}_t)d \mo^{s,x}_t = \int_s^{s'} \frac{\cos(\pi \mo^{s,x}_t)}{\sqrt{1-t}}dB_t,\]
	which is a local martingale bounded in $L^1$ as long as $t\in[0,1-\epsilon]$. Hence it is an actual martingale on $[0,1-\epsilon]$ for $0<\epsilon<1$ arbitrary, and thus also a martingale on $[0,1)$. On the other hand
	\begin{align*}
	\partial_tf(t,z)+\frac{\so(t,z)}{2}\partial^2_{zz}f(t,z)&=	\partial_tf(t,z)+\frac{\sin^2(\pi z)}{2\pi^2(1-t)}\partial^2_{zz}f(t,z)\\ &=\frac{1}{2(1-t)}+\frac{\sin^2(\pi z)}{2\pi^2(1-t)}\partial_{z}\left(\frac{\pi\cos(\pi z)}{\sin(\pi z)}\right)\\
		&= \frac{1}{2(1-t)} - \frac{1}{2(1-t)}\{\sin^2(\pi z)+\cos^2(\pi z)\} \\
		& = 0. 
	\end{align*}
	Hence \[L_t=L_s + \int_s^{t} \frac{\cos(\pi \mo^{s,x}_u)}{\sqrt{1-u}}dB_u,\]
	and we conclude.
\end{proof}

Associated to the martingale $\mo$ we define its cost
\begin{align}\label{eq:def_v}
\vo(s,x):=\frac{1}{2}\mathbb E\left[\int_s^1 \left \{\frac{\sin^2(\pi\mo^{s,x}_t)}{\pi^2(1-t)}- 2 \log\frac{\sin(\pi\mo^{s,x}_t)}{\pi\sqrt{1-t}} -1\right \} dt  \right ].
\end{align}

\begin{lemma}\label{lem_formula_v}
We have $$\vo(s,x)= \frac{x-x^2-1+s}{2}-(1-s)\log\left(\frac{\sin(\pi x)}{\pi\sqrt{1-s}} \right ).$$	
\end{lemma}

\begin{proof}
	Clearly
	\[\frac{1}{2}\mathbb E\left[\int_s^1 \left \{\frac{\sin^2(\pi\mo^{s,x}_t)}{\pi^2(1-t)}-1\right \} dt  \right ]=\frac{1}{2}\mathbb E[(\mo^{s,x}_1)^2-(\mo^{s,x}_s)^2 +s-1]=\frac{x-x^2-1+s}{2}.  \]
	On the other hand, owing to Lemma \ref{lem_log_is_mart} we have
	\[ \mathbb E\left[\int_s^1  \log\frac{\sin(\pi\mo^{s,x}_t)}{\pi\sqrt{1-t}} dt  \right ]  = \int_s^1 \mathbb E\left[ \log\frac{\sin(\pi\mo^{s,x}_t)}{\pi\sqrt{1-t}}\right ] dt  = (1-s)\log\frac{\sin(\pi x)}{\pi\sqrt{1-s}},\]
	where the exchange between expectation and integral is justified by the fact that $\log(1-t)$ is integrable on $[s,1]$ and $\log (\sin(\pi \mo^{s,x}_t))$ has a constant sign.
	We conclude.
\end{proof}

\begin{lemma}\label{lem:second_var}
$\so^2(t,x)$ is the unique minimizer of the function
$$[0,\infty)\ni\Sigma \mapsto \Sigma \partial^2_{xx}\vo(t,x)+\Sigma -\log\Sigma .$$
\end{lemma}

\begin{proof} The function in the statement is strictly convex so its unique minimizer can be found by equating its derivative to zero. This means $\partial^2_{xx}\vo(t,x)+1 = 1/ \Sigma$. So we only have to check that 
$$\so^2(t,x):= \frac{\sin^2(\pi x)}{\pi^2(1-t)}= \frac{1}{1+\partial^2_{xx}\vo(t,x)}.$$
But from Lemma \ref{lem_formula_v} we have
$$\partial^2_{xx}\vo(t,x) = \frac{\pi^2(1-t)}{\sin^2(\pi x)} -1,$$
which concludes the proof.
\end{proof}

\begin{lemma}\label{lem_HJB}
	On $(0,1)\times(0,1)$ we have
	\[\partial_t \vo(t,x) +\frac{1}{2}\inf_{\Sigma\geq 0}\, \{\Sigma\, \partial^2_{xx}\vo(t,x)+\Sigma -\log\Sigma-1\}=0. \]
\end{lemma}

\begin{proof}
	By \eqref{eq:def_v} we have that
	\[t\mapsto \vo(t,\mo^{0,x})+\frac{1}{2}\int_0^t\{ \so^2(\mo^{0,x}_u) - \log\so^2(\mo^{0,x}_u) -1\}du \]
	is a martingale. Since $\mo^{0,x}$ does not leave $(0,1)$ this means that
	\[\partial \vo(t,z) +\frac{1}{2}\so^2(t,z)\partial_{zz}\vo(t,z) + \frac{1}{2}\{\so^2(t,z)-\log\so^2(t,z)-1\}=0.\]
	But then by  Lemma \ref{lem:second_var} the l.h.s.\ above is equal to $\partial_t \vo(t,x) +\frac{1}{2}\inf_{\Sigma\geq 0}\, \{\Sigma\, \partial^2_{xx}\vo(t,x)+\Sigma -\log\Sigma-1\}$.
\end{proof}

\begin{lemma}\label{lem:various_vs}
	For $s\in[0,1)$ and $x\in(0,1)$ let 
	\[\tilde v(s,x):= \frac{x-x^2-1+s}{2}-(1-s)\log\left(\frac{\sqrt{x(1-x)}}{\sqrt{1-s}} \right ),\]
	and
	\[v(s,x):=\inf\frac{1}{2}\mathbb E\left[\int_s^1 \left \{\sigma_t^2-  \log(\sigma_t^2) -1\right \} dt  \right ],
\]
where the infimum runs over all martingales $M$ on $[s,1]$, starting at $x$, satisfying $dM_t=\sigma dW_t$, and such that $M_1\in\{0,1\}$.

	Then $v(\cdot,\cdot)\geq \tilde v(\cdot,\cdot)$ and there is $\delta\in (0,\infty)$ such that
	\begin{align}
	\label{eq:abschaetzung_v}	
\left |2\tilde v(s,x)-\vo(s,x)\right | \leq \frac{|x-x^2-1+s+(1-s)\log(1-s)|}{2}+\delta\cdot(1-s).
	\end{align}
\end{lemma}

\begin{proof}
	By Jensen's inequality 
	\[\mathbb E\left[\int_s^1 \left \{\sigma_t^2-  \log(\sigma_t^2) -1\right \} dt  \right ]\geq \mathbb E\left[\int_s^1 \sigma_t^2 dt \right ] -  (1-s)\log\left(\frac{\mathbb E\left[\int_s^1 \sigma_t^2 dt \right ]}{1-s}\right ) - (1-s) .\]
	If $M$ participates in the infimum defining $v(s,x)$, then $M_1\sim B(x)$ and so $\mathbb E\left[\int_s^1 \sigma_t^2 dt \right ]=\mathbb E[M_1^2-M_s^2] = x -x^2=x(1-x)$. This shows that $v\geq \tilde v$. On the other hand 
	\begin{align*}
		2\tilde v(s,x)-\vo(s,x)&= \frac{x-x^2-1+s+(1-s)\log(1-s)}{2}-(1-s)\log\left (\frac{\pi x(1-x)}{\sin(\pi x)} \right ),
	\end{align*}
and we obtain \eqref{eq:abschaetzung_v} by noticing that
\[0<\inf_{x\in[0,1]}\frac{\pi x(1-x)}{\sin(\pi x)}\leq \sup_{x\in[0,1]}\frac{\pi x(1-x)}{\sin(\pi x)} <+\infty. \]
\end{proof}

\begin{lemma}\label{lem_veri}
	Let $M$ be feasible for our minimization problem (started from $x_0$ at time 0), and denote by $\Sigma_t$ the density of its quadratic variation. 
Suppose that $\mathbb E[\int_0^1\log \Sigma_t dt]$ is finite, that $M_t$ does not leave the interval $(0,1)$ until time $t=1$, and that the process $t\mapsto \int_0^t (1-s)\cot(\pi M_s)dM_s$ is a martingale on $[0,1]$. 	Then the process
	$$t\mapsto R_t^M:=\vo(t,M_t)+\frac{1}{2}\int_0^t\{\Sigma_s-\log\Sigma_s-1\}ds
$$
is a submartingale on $[0,1]$, and it is actually a martingale  on $[0,1]$ for $M=\mo^{0,x_0}$ and $\Sigma_t=\so(t,\mo^{0,x_0}_t)$.
\end{lemma}

\begin{proof}
	By Lemma \ref{lem_HJB} we have
	\[\partial_t \vo(t,M_t)+ \frac{1}{2} \{\Sigma_t \partial^2_{xx}\vo(t,M_t)+\Sigma_t -\log\Sigma_t-1\} \geq 0,\]
	from which, thanks to Ito formula, the local submartingale property of $R^M$ follows. To check that this process is an actual submartingale, we need to establish that 
	 $\int_0^t \partial_x\vo(s,M_s)dM_s$ is a martingale. But
	 \[\int_0^t \partial_x\vo(s,M_s)dM_s=\int_0^t\left\{\frac{1-2M_s}{2}-\pi(1-s)\cot(\pi M_s)  \right\}dM_s,\]
	 which by assumption is indeed a martingale.
	 
	On the other hand, for $M=\mo^{0,x_0}$ and $\Sigma_t=\so(t,\mo^{0,x_0}_t)$ we have 
	\[\partial_t \vo(t,\mo^{0,x_0}_t)+ \frac{1}{2} \{\so(\mo^{0,x_0}_t)^2 \partial^2_{xx}\vo(t,\mo^{s,x_0}_t)+\so(\mo^{0,x_0}_t)^2 -\log\so(\mo^{0,x_0}_t)^2-1\} = 0,\]
	so by the same token the process $R^{\mo}$ is a local martingale.
	This time around
	 \[\int_0^t \partial_x\vo(s,\mo^{0,x_0}_s)d\mo^{0,x_0}_s=\int_0^t\left\{\frac{1-2\mo^{0,x_0}_s}{2}-\pi(1-s)\cot(\pi \mo^{0,x_0}_s)  \right\}\frac{\sin(\pi \mo^{0,x_0}_s)}{\pi\sqrt{1-s}}dB_s,\]
	 which is a martingale. We conclude that $R^{\mo}$ is a martingale.
\end{proof}

We can now carry on the verification argument, showing the optimality of $\mo$:

\begin{proof}
If $M$ fulfils the assumptions in Lemma	 \ref{lem_veri},
	then we have
	\begin{align*}
	\mathbb E\left [
\frac{1}{2}\int_0^1\{\Sigma_s-\log\Sigma_s-1\}ds	\right ] &=
	\mathbb E\left [
\vo(1,M_1)+\frac{1}{2}\int_0^1\{\Sigma_s-\log\Sigma_s-1\}ds	\right ]\\
&\geq 	\vo(0,x_0) \\
&= 	\mathbb E\left [
\vo(1,\mo^{0,x_0}_1)+\frac{1}{2}\int_0^1\{\so(\mo^{0,x_0}_s)^2-\log\so(\mo^{0,x_0}_s)^2-1\}ds	\right ]\\
&= \mathbb E\left [
\frac{1}{2}\int_0^1\{\so(\mo^{0,x_0}_s)^2-\log\so(\mo^{0,x_0}_s)^2-1\}ds	\right ],
	\end{align*}
showing that $\mo^{0,x_0}$ is better than $M$ (i.e.\ $M$ is sub-optimal compared to $\mo^{0,x_0}$). 

Let now $M$ be feasible and wlog such that $\mathbb E[\int_0^1\log \Sigma_t dt]$ is finite. For $\epsilon$ small we define $\tau^\epsilon:=\inf \{t:M_t\leq \epsilon \text{ or }M_t\geq 1-\epsilon\}$. We then define the martingale $M^\epsilon$ as follows: for $t\leq t^\epsilon:=\min\{1-\epsilon,\tau^{\epsilon}\}$ we set $M^\epsilon_t:=M_{t}$, whereas for $t\in(t^{\epsilon},1]$ we set $M^\epsilon_t:= \mo^{t^{\epsilon},M_{t^\epsilon}}_t$. Hence $M^\epsilon$ is just the continuous pasting (concatenation) of $M$ and $\mo$ at time $t^\epsilon$. One easily checks that $M^\epsilon$ fulfils the assumptions in Lemma \ref{lem_veri}, for all $\epsilon>0$. If $\Sigma^\epsilon$ is the density of the quadratic variation of $M^\epsilon$, then
\begin{align*}
&\frac{1}{2}\mathbb E\left [\int_0^1\{\Sigma^\epsilon_t - \log\Sigma^\epsilon_t -1\} dt\right ]	\\ = & 
\frac{1}{2}\mathbb E\left [\int_0^{t^\epsilon}\{\Sigma_t - \log\Sigma_t -1\} dt\right ] + \frac{1}{2}\mathbb E\left [\int^1_{t^\epsilon}\{\so^2(t,\mo_t^{t^\epsilon,M_{t^\epsilon} }) - \log\so^2(t,\mo_t^{t^{\epsilon},M_{t^\epsilon}}) -1\} dt \right ]\\
=& \frac{1}{2}\mathbb E\left [\int_0^{t^\epsilon}\{\Sigma_t - \log\Sigma_t -1\} dt\right ] + \mathbb E[\vo(t^{\epsilon}, M_{t^\epsilon})].
\end{align*}
For the time being let us assume that
\begin{align}
\label{eq:error_to_zero}
	\limsup_{\epsilon\searrow 0}\mathbb E[\vo(t^{\epsilon}, M_{t^\epsilon})]=0.
\end{align}
Then, since $t^\epsilon\nearrow 1$ as $\epsilon\searrow 0$, this would entail
\[\limsup_{\epsilon \to 0} \frac{1}{2}\mathbb E\left [\int_0^1\{\Sigma^\epsilon_t - \log\Sigma^\epsilon_t -1\} dt\right ] \leq \frac{1}{2}\mathbb E\left [\int_0^{1}\{\Sigma_t - \log\Sigma_t -1\} dt\right ],\]
so that $M$ is sub-optimal compared to  $\mo^{0,x_0}$. Hence to finish the proof we have to establish \eqref{eq:error_to_zero}. For this purpose, owing to \eqref{eq:abschaetzung_v}, it is enough to check that
\begin{align} \label{eq:error_to_zero2}
	\limsup_{\epsilon\searrow 0}\mathbb E[\tilde v(t^{\epsilon}, M_{t^\epsilon})]\leq 0,
\end{align}
with $\tilde v$ defined in Lemma \ref{lem:various_vs}. Taking $v$ from that same lemma, we find
\begin{align*}
&\frac{1}{2}\mathbb E\left [\int_0^{1}\{\Sigma_t - \log\Sigma_t -1\} dt\right ] \\
=& \frac{1}{2}\mathbb E\left [\int_0^{t^\epsilon}\{\Sigma_t - \log\Sigma_t -1\} dt\right ] +\frac{1}{2}\mathbb E\left [\int_{t^\epsilon}^1\{\Sigma_t - \log\Sigma_t -1\} dt\right ] 	
\\ \geq & \frac{1}{2}\mathbb E\left [\int_0^{t^\epsilon}\{\Sigma_t - \log\Sigma_t -1\} dt\right ] + \mathbb E[v(t^\epsilon,M_{t^\epsilon})] \\
\geq & \frac{1}{2}\mathbb E\left [\int_0^{t^\epsilon}\{\Sigma_t - \log\Sigma_t -1\} dt\right ] + \mathbb E[\tilde v(t^\epsilon,M_{t^\epsilon})],
\end{align*}
as $v\geq \tilde v$. Since $\mathbb E[\int_0^1\log \Sigma_t dt]$ is finite, we can derive \eqref{eq:error_to_zero2} from this.
\end{proof}

We close this part remarking on the uniqueness of optimizers to Problem \eqref{eq:max_ent_prob}: As the previous proofs show, the only way for $\Sigma$ to be optimal is by making \[\partial_t \vo(t,M_t)+ \frac{1}{2} \{\Sigma_t \partial^2_{xx}\vo(t,M_t)+\Sigma_t -\log\Sigma_t-1\} \]
be equal to zero.  By Lemma \ref{lem:second_var} this is only achieved by the Aldous martingale.

\bibliographystyle{abbrv}
\bibliography{../MBjointbib/joint_biblio}

\end{document}